\documentclass[reqno,A4paper]{amsart}

\usepackage{amsmath}
\usepackage{amssymb}
\usepackage{amsthm}
\usepackage{enumerate}
\usepackage{mathrsfs} 
\usepackage{eqlist}
\usepackage{array}
\usepackage[utf8]{inputenc}
\usepackage{hyperref}
\hypersetup{
colorlinks=true,
urlcolor=blue,
citecolor=blue}
\theoremstyle{plain}
\newtheorem{theorem}{Theorem}[section]

\newtheorem{lemma}{Lemma}[section]
\newtheorem{corollary}{Corollary}[theorem]

\theoremstyle{definition}

\newtheorem{remark}{\textup{Remark}} 

\numberwithin{equation}{section}

\begin{document}

\title[On some congruences using multiple harmonic sums of length three and four \LaTeX 2e]%
{On some congruences using multiple harmonic sums of length three and four}
\author[WALID KEHILA]%
{WALID KEHILA}

\newcommand{\acr}{\newline\indent}

\address{\llap{*\,}University of Science and Technology\acr
                   Houari Boumedienne\acr
                   USTHB\acr
                   Faculty of Mathematics, P.B. 32 El Alia, 16111 \acr
                   Algiers, Algeria}
\email{wkehilausthb@gmail.com; wkehila@usthb.dz}



\begin{abstract}
In the present paper, we determine the sums $\sum_{j=1}^{p-1}\frac{H_j^{(s_1)}H_j^{(s_3)}}{j^{s_2}}$ and $\sum_{j=1}^{p-1}\frac{H_j^{(s_1)}H_j^{(s_3)}H_j^{(s_4)}}{j^{s_2}}$ modulo $p$ and modulo $p^2$ in certain cases. This is done by using multiple harmonic sums of length three and four, as well as, many other results. In addition, We recover three congruences conjectured by Z.-W Sun and solved later by the author himself and R. Meštrović.
\end{abstract}

\maketitle

\section{Introduction }

Multiple Harmonic  Sums (MHS) are defined by
$$H(s_1,\dots,s_k;n):=\sum_{1\le  j_1 < j_2 < \cdots < j_k \le {n}} \frac{1}{j_1^{s_1}\dots j_k^{s_k}},$$
with the conventions $H(s_1,\dots,s_k;r)=0$ for $r=0,\dots,k-1$, and, $H(\emptyset;0)=1$.
They satisfy the following recurrence relation \cite{tauraso}
$$H(s_1,\dots,s_k;n)=\sum_{k=1}^{n}\frac{1}{k^{s_k}}H(s_1,\dots,s_{k-1};k-1).$$
In the case when $s_1=\dots=s_k=s$, these sums are called the homogeneous multiple harmonic sums, and denoted
$$H(\{s\}^k;n):=H(\underbrace{s,\dots,s}_{\text{k times}} ;n)=\sum_{1\le  j_1 < j_2 < \cdots < j_k \le {n}} \frac{1}{(j_1\dots j_k)^s}.$$
When $k=1$, we find the sequence of generalized harmonic numbers, we may denote it
$$H_n^{(s)}=\sum_{k=1}^{n}\frac{1}{k^s},$$
note that the superscript is omitted in the case $s=1$.
 
When $n=p-1$ we may simplify notations as follows 
$$H(s_1,\dots,s_k):=\sum_{1\le  j_1 < j_2 < \cdots < j_k \le {p-1}} \frac{1}{j_1^{s_1}\dots j_k^{s_k}} \text{\quad and \quad }H(s):=H_{p-1}^{(s)}=\sum_{k=1}^{p-1}\frac{1}{k^s}.$$
In \cite{sun1} Sun has proposed some conjectures namely (Conjecture 1.1 and Conjecture 1.2); later, in \cite{sun2} he proved (Conjecture 1.2). Meštrović \cite{mistrovic2}, on the other hand, established the second part of Conjecture 1.1 using congruences of (MHS) of length three that can be found in \cite{zhao}.
  
In the present paper, we may unify the proof of these conjectures, as well as, proving many other congruences.

Now, we give some results which will be using in the present paper.
\begin{theorem}\cite{zhao} \label{cancelation}
Let $s$ and $l$ be two positive integers. Let $p$ be an odd prime such that $p\geq l+2$ and $p-1$ divides non of $ls$ and $ks+1$ for $k=1,\dots,l$. Then
\begin{equation*}
H(\{s\}^l)\equiv \begin{cases}
0&\pmod {p^2} \text{\quad for \quad }ls-1 \text{ even}\\
0&\pmod {p} \text{ \quad for \quad}ls-1 \text{ odd}.
\end{cases}
\end{equation*}
In particular, when $p\geq ls+3$, the above is always true.
\end{theorem}  
\begin{theorem}\cite{zhao}\label{cancelation2}
Let $s$ and $k$ be two non-negative integers. Suppose also that $p\geq sk+3$, then
\begin{equation*}
H(\{s\}^k) \equiv\begin{cases}
(-1)^k\frac{s(sk+1)p^2}{2(sk+2)}B_{p-sk-2}&\pmod {p^3} \text{\quad for }ks \text{ odd}\\
(-1)^{k-1} \frac{sp}{sk+1}B_{p-sk-1}&\pmod {p^{2}} \text{\quad for }ks \text{ even},
\end{cases} 
\end{equation*}
where $(B_n)_{n\in \mathbb{N}}$ is the sequence of Bernoulli numbers.
\end{theorem}
\begin{theorem}\cite{zhao} \label{weighttwo}
Let $s_1$,$s_2$ be two positive integers and $p\geq 3$. Let $s_1\equiv m,  s_2\equiv n\pmod {p-1}$ where $0\leq m, n\leq p-2$. If m,n$\geq 1$ then 
\begin{equation*}
H(s_1,s_2)\equiv\begin{cases}
\frac{(-1)^{n}}{m+n}{m+n \choose m}B_{p-m-n}&\pmod p  \text{\quad for\quad } p\geq m+n\\
0& \pmod p \text{\quad for\quad} p< m+n.
\end{cases}
\end{equation*}
Furthermore, when $s_1+s_2$ is even and $p> s_1+s_2+1$
\begin{align*}
H(s_1,s_2)&\equiv p\left((-1)^{s_1}s_2{s_1+s_2+1\choose s_1}-(-1)^{s_1}s_1{s_1+s_2+1\choose s_2}-s_1-s_2\right)\\
&\quad \times \frac{B_{p-s_1-s_2-1}}{2(s_1+s_2+1)}\pmod {p^{2}}.
\end{align*}
\end{theorem}

\begin{theorem}\cite{zhao}\cite{tauraso1}\label{weightthree} Suppose that $w:=s_1+s_2+s_3$ is odd, then for primes $p> w$ we have
\begin{equation*}
H(s_1,s_2,s_3)\equiv \left((-1)^{s_1}{w\choose s_1}-(-1)^{s_3}{w\choose s_3}\right)\frac{B_{p-w}}{2w}\pmod p.
\end{equation*}
In particular, when $s_1=s_3$, and $s_2$ is odd, we have
\begin{equation*}
H(s_1,s_2,s_1)\equiv 0\pmod p.
\end{equation*}
\end{theorem}
\section{Congruences using MHS of length three and two}

We start by establishing the following theorem which is motivated by a paper of Meštrović \cite{mistrovic2}.
\begin{theorem}\label{principal}For all positive integers $s_1,s_2,s_3$, we have the following 
\begin{equation}\label{principalequation}
\sum_{j=1}^{p-1}\frac{H_j^{(s_1)}H_j^{(s_3)}}{j^{s_2}}=-H(s_1,s_2,s_3)+H(s_3,s_1+s_2)+H(s_1+s_2+s_3)+H(s_3)H(s_1,s_2).
\end{equation}
Equivalently, we have
\begin{equation}\label{principalequation2}
\sum_{j=1}^{p-1}\frac{H_j^{(s_1)}H_j^{(s_3)}}{j^{s_2}}=H(s_1,s_3,s_2)+H(s_3,s_1,s_2)+H(s_3,s_1+s_2)+H(s_1+s_3,s_2)+H(s_1,s_2+s_3)+H(s_1+s_2+s_3).
\end{equation}
\end{theorem}
\begin{proof}
For $j=1,\dots,p-1$, we have
$$\frac{1}{(j+1)^{s_3}}+\dots+\frac{1}{(p-1)^{s_3}}= -H^{(s_3)}_{j}+H(s_3). $$
We obtain the following
\begin{align*}
H(s_1,s_2,s_3)=&\sum_{1\le  i < j < k \le {p-1}} \frac{1}{i^{s_1}j^{s_2} k^{s_3}}=\sum_{j=2}^{p-1}\frac{1}{j^{s_2}}\sum_{i=1}^{j-1}\frac{1}{i^{s_1}}\sum_{k=j+1}^{p-1}\frac{1}{k^{s_3}}\\
=&\sum_{j=2}^{p-1}\frac{1}{j^{s_2}}(1+\frac{1}{2^{s_1}}+\dots+\frac{1}{(j-1)^{s_1}})(\frac{1}{(j+1)^{s_3}}+\dots+\frac{1}{(p-1)^{s_3}})\\
=&\sum_{j=2}^{p-1}\frac{1}{j^{s_2}}(1+\frac{1}{2^{s_1}}+\dots+\frac{1}{(j-1)^{s_1}})(-H^{(s_3)}_{j}+H(s_3))\\
=&-\sum_{j=1}^{p-1}\frac{1}{j^{s_2}}(H_j^{(s_1)}-\frac{1}{j^{s_1}})H_j^{(s_3)}+H(s_3)\sum_{j=1}^{p-1}\frac{1}{j^{s_2}}H_{j-1}^{(s_1)}\\
=&-\sum_{j=1}^{p-1}\frac{H_j^{(s_1)}H_j^{(s_3)}}{j^{s_2}}+\sum_{j=1}^{p-1}\frac{1}{j^{s_1+s_2}}(H_{j-1}^{(s_3)}+\frac{1}{j^{s_{3}}})+H(s_3)\sum_{j=1}^{p-1}\frac{1}{j^{s_2}}H_{j-1}^{(s_1)}\\
=&-\sum_{j=1}^{p-1}\frac{H_j^{(s_1)}H_j^{(s_3)}}{j^{s_2}}+H(s_3,s_1+s_2)+H(s_1+s_2+s_3)+H(s_3)H(s_1,s_2).
\end{align*}
Therefore
$$\sum_{j=1}^{p-1}\frac{H_j^{(s_1)}H_j^{(s_3)}}{j^{s_2}}=-H(s_1,s_2,s_3)+H(s_3,s_1+s_2)+H(s_1+s_2+s_3)+H(s_3)H(s_1,s_2).$$
Equation \eqref{principalequation2} follows immediately by the shuffle product relation
$$H(s_3)H(s_1,s_2)=H(s_3,s_1,s_2)+H(s_1,s_3,s_2)+H(s_1,s_2,s_3)+H(s_3+s_1,s_2)+H(s_1,s_3+s_2).$$
\end{proof}
\begin{remark}\label{principalremark}
One can easily observe that the previous proof  does not depend on $p-1$, so for all positive integers $n$, we have
$$\sum_{j=1}^{n}\frac{H_j^{(s_1)}H_j^{(s_3)}}{j^{s_2}}=-H(s_1,s_2,s_3;n)+H(s_3,s_1+s_2;n)+H(s_1+s_2+s_3;n)+H(s_3;n)H(s_1,s_2;n).$$
and 
\begin{equation*}
\begin{split}
\sum_{j=1}^{n}\frac{H_j^{(s_1)}H_j^{(s_3)}}{j^{s_2}}=&H(s_1,s_3,s_2;n)+H(s_3,s_1,s_2;n)+H(s_3,s_1+s_2;n)\\
&+H(s_1+s_3,s_2;n)+H(s_1,s_2+s_3;n)+H_n^{(s_1+s_2+s_3)}.
\end{split}
\end{equation*}
\end{remark}  
\subsection{Congruences mod $p$}
From Theorem \ref{principal} one can see that determining $\sum_{j=1}^{p-1}\frac{H_j^{(s_1)}H_j^{(s_3)}}{j^{s_2}}$ depends on determining MHS of length three and two, Using Theorems \ref{weightthree} and \ref{weighttwo} we obtain the following result.
\begin{theorem}\label{congmodp} Suppose that $w$ is odd then
\begin{equation}
\sum_{j=1}^{p-1}\frac{H_j^{(s_1)}H_j^{(s_3)}}{j^{s_2}}\equiv \left[\frac{(-1)^{s_1+1}}{2w}{w\choose s_1}+\frac{(-1)^{s_3}+2(-1)^{s_1+s_2}}{2w}{w\choose s_3} \right] B_{p-w} \pmod p.
\end{equation}
\end{theorem}
\begin{proof}
From the first part of Theorem \ref{principal}, and in light of Theorems \ref{weightthree} and \ref{weighttwo}, we have modulo $p$
\begin{align*}
\sum_{j=1}^{p-1}\frac{H_j^{(s_1)}H_j^{(s_3)}}{j^{s_2}}&\equiv -\left(  (-1)^{s_1}{w\choose s_1}-(-1)^{s_3}{w\choose s_3}\frac{B_{p-w}}{2w}\right) +\frac{(-1)^{s_1+s_2}}{w}{w \choose s_3}B_{p-w}\\
&\equiv \left[\frac{(-1)^{s_1+1}}{2w}{w\choose s_1}+\frac{(-1)^{s_3}+2(-1)^{s_1+s_2}}{2w}{w\choose s_3} \right] B_{p-w}.
\end{align*}

\end{proof}
\begin{corollary}Let $s\geq 1$ and $r$ is odd, then for $p> 2s+r$, we have
\begin{equation}
\sum_{j=1}^{p-1}\frac{(H_j^{(s)})^2}{j^{r}}\equiv (-1)^{s+r}{2s+r\choose s}\frac{B_{p-2s-r}}{2s+r}  \pmod p.
\end{equation}
For all $s\geq 1$ and any prime $p\geq 3$ such that $p-1$ does not divide $3s$, in particular when $p\geq 3s+3$, we have
\begin{equation}\label{conjecture1}
\sum_{j=1}^{p-1}\frac{(H_j^{(s)})^2}{j^{s}}\equiv {3s\choose s}\frac{B_{p-3s}}{3s} \pmod p.
\end{equation}
Moreover, when $s$ is even, then
\begin{equation}
\sum_{j=1}^{p-1}\frac{(H_j^{(s)})^2}{j^{s}}\equiv 0 \pmod p.
\end{equation}
\end{corollary}

\begin{proof}
For the first part, we apply Theorems \ref{cancelation} and \ref{weightthree}.

For the second part, we apply Theorems \ref{cancelation} and \ref{weighttwo}; hence
the last part yields immediately since Bernoulli numbers vanishes for odd values. This would solve the first part of Sun Conjecture see (\cite{sun1} Conjecture 1.2) and (\cite{sun2} Theorem 1.2).  
\end{proof}
\subsection{Some special cases}
In this section, we shall prove some congruences modulo $p$ that are not covered by Theorem \ref{congmodp}.
\begin{lemma}\cite{hoffman} Theorem 7.2, \cite{zhao} page 94 and Proposition 3.8

For $p\geq 11$ we have
\begin{equation}
\frac{1}{3}H(2,3,1)\equiv-\frac{1}{2}H(3,2,1)\equiv-H(3,1,2)\equiv -\frac{1}{2}H(1,4,1)\equiv H(4,1,1)\equiv -\frac{1}{6}B_{p-3}^2{\pmod {p}}.
\end{equation}
For $p\geq 7$ we have
\begin{equation}
H(1,2,2)\equiv -\frac{3}{2}B_{p-5}{\pmod {p}}.
\end{equation}
For $p\geq 17$ we have
\begin{equation}
H(5,3,4)\equiv H(4,3,5)\equiv 0{\pmod {p}}.
\end{equation}
\end{lemma}

\begin{theorem}
For $p\geq 11$ we have
\begin{equation}
2\sum_{j=1}^{p-1}\frac{H_j^{(2)}H_j}{j^3}\equiv -3\sum_{j=1}^{p-1}\frac{H_j^{(3)}H_j}{j^2}\equiv -6\sum_{j=1}^{p-1}\frac{H_j^{(3)}H_j^{(2)}}{j} \equiv-3\sum_{j=1}^{p-1}\frac{(H_j)^2}{j^4} \equiv  6\sum_{j=1}^{p-1}\frac{H_j^{(4)}H_j}{j}\equiv B_{p-3}^2\pmod p,
\end{equation}
\begin{equation}
\sum_{j=1}^{p-1}\frac{H_jH_j^{(2)}}{j^{2}}\equiv -\frac{1}{2}B_{p-5} \pmod p.
\end{equation}
For $p\geq 17$ we have
\begin{equation}
\sum_{j=1}^{p-1}\frac{H_j^{(5)}H_j^{(4)}}{j^{3}}\equiv 0 \pmod p.
\end{equation}
\end{theorem}
\begin{proof}
We shall prove only the last one, since the others are proved in a similar manner.

From Theorems \ref{principal}, \ref{weighttwo}  and the previous Lemma we find 
$$\sum_{j=1}^{p-1}\frac{H_j^{(5)}H_j^{(4)}}{j^{3}}\equiv -H(5,3,4)+ H(4,8)\equiv0+ \frac{165}{4} B_{p-12}=0 \pmod p.$$
\end{proof}
\subsection{Congruences mod $p^2$}
In this section, we prove some congruences modulo $p^2$ in certain cases only, since determining (MHS) of length three seems to be a much more involved problem.
\begin{lemma}\cite{zhao} Proposition 3.7

For $p\geq 7$ we have
\begin{equation}
\frac{10}{9}H(1,2,1)\equiv\frac{5}{3} H(2,1,1)\equiv \frac{10}{11}H(1,1,2)\equiv pB_{p-5}{\pmod {p^2}},
\end{equation}
\begin{equation}
H(1,3,1)\equiv 0{\pmod {p^2}},
\end{equation}
\begin{equation}
H(4,1)\equiv -B_{p-5}{\pmod {p^2}}.
\end{equation}
\end{lemma}

\begin{corollary}For $p\geq 7$ we have
\begin{equation}
\sum_{j=1}^{p-1}\frac{H_j^{2}}{j^{2}}\equiv H(4) \equiv  \frac{4}{5}pB_{p-5}{\pmod {p^2}},
\end{equation}
\begin{equation}
\sum_{j=1}^{p-1}\frac{H_j^{(2)}H_j}{j}\equiv -\frac{7}{10}pB_{p-5}{\pmod {p^2}},
\end{equation}
\begin{equation}
\sum_{j=1}^{p-1}\frac{H_j^{2}}{j^{3}}\equiv B_{p-5}{\pmod {p^2}}.
\end{equation}
\end{corollary}
\begin{proof}
From Theorem \ref{weighttwo} and the previous Lemma we obtain
$$H(1,2,1)\equiv \frac{9}{10}pB_{p-5}\equiv H(1,3){\pmod {p^2}}.$$
Substituting in \eqref{principalequation} and using Theorem \ref{cancelation2} we find
$$\sum_{j=1}^{p-1}\frac{H_j^{2}}{j^{2}}\equiv H(4) \equiv  \frac{4}{5}pB_{p-5}{\pmod {p^2}}.$$
This establishes the second part of Sun Conjecture see (\cite{sun1} conjecture 1.1) proved by Meštrović see \cite{mistrovic2}.  

From Theorems \ref{weighttwo}, \ref{cancelation2}, and the previous Lemma we get 
$$\sum_{j=1}^{p-1}\frac{H_j^{(2)}H_j}{j}\equiv -\frac{3}{5}pB_{p-5}-\frac{9}{10}pB_{p-5}+\frac{4}{5}pB_{p-5}\equiv -\frac{7}{10}pB_{p-5}{\pmod {p^2}}.$$
From Equation \eqref{principalequation} and the previous Lemma we obtain
$$\sum_{j=1}^{p-1}\frac{H_j^{2}}{j^{3}}\equiv H(1,4){\pmod {p^2}}.$$
Since
$$H(4,1)\equiv -B_{p-5}{\pmod {p^2}},$$
we have from Theorem 3.2 in \cite{zhao}
$$H(4,1)+H(1,4)\equiv-B_{p-5}+H(1,4) \equiv-\frac{5}{6}pB_{p-6}=0{\pmod {p^2}}.$$
\end{proof}
\begin{corollary}\label{conjecture2}Suppose $s$ is even, then for $p> 3s+1$   we have
\begin{equation}\label{conjecture2equa}
\sum_{j=1}^{p-1}\frac{(H_j^{(s)})^2}{j^{s}}\equiv  \left[{3s+1\choose s-1}+\frac{s}{2} \right] p\frac{B_{p-3s-1}}{3s+1} {\pmod {p^2}}.
\end{equation}
\end{corollary}
\begin{proof}
Using the identity
$${n\choose k+1}=\frac{n-k}{k+1}{n\choose k},$$
and from the first part of Theorem \ref{principalequation}, we find modulo $p^2$
\begin{align*}
\sum_{j=1}^{p-1}\frac{(H_j^{(s)})^2}{j^{s}}&\equiv -H(s,s,s)+H(s,2s)+H(3s) \\
&\equiv \left[ -s+\frac{2s{3s+1\choose s}-s{3s+1\choose 2s}-3s}{2}  +3s\right] p\frac{B_{p-3s-1}}{3s+1}\\
&=\left[2+\frac{{3s+1\choose s}-\frac{s}{s+1}{3s+1\choose s}-3}{2} \right] sp\frac{B_{p-3s-1}}{3s+1}\\
&=\left[ 1+\frac{{3s+1\choose s}}{s+1} \right] \frac{sp}{2}\frac{B_{p-3s-1}}{3s+1}\\
&=\left[{3s+1\choose s-1}+\frac{s}{2} \right] p\frac{B_{p-3s-1}}{3s+1}.
\end{align*}
Congruences \eqref{conjecture1} and \eqref{conjecture2equa} establishes  Conjecture 1.2 (see \cite{sun1} and \cite{sun2} Theorem 1.2) with a slightly different approach.
\end{proof}

\section{Some special congruences using MHS of length four}
In this section, we prove some congruences using some results on (MHS) of length four.
\begin{theorem}\label{principallengthfour}
The following equality holds true
\begin{multline}\label{principallengthfoureq}
-\sum_{j=1}^{p-1}\frac{H_j^{(s_1)}H_j^{(s_3)}H_j^{(s_4)}}{j^{s_2}}=H(s_4)\sum_{j=1}^{p-1}\frac{H_j^{(s_1)}H_j^{(s_3)}}{j^{s_2}}+H(s_1,s_3,s_2,s_4)+H(s_3,s_1,s_2,s_4)\\
+H(s_3,s_1+s_2,s_4)+H(s_1+s_3,s_2,s_4)+H(s_1,s_2+s_3,s_4)+H(s_1+s_2+s_3,s_4).
\end{multline}

\end{theorem}
\begin{proof}From Remark \ref{principalremark} and Equation \eqref{principalequation2} we find
\begin{align*}
\sum_{l=1}^{p-1}\frac{1}{l^{s_4}}\sum_{j=1}^{l-1}\frac{H_j^{(s_1)}H_j^{(s_3)}}{j^{s_2}}=&\sum_{l=1}^{p-1}\frac{1}{l^{s_4}}\bigg[  H(s_1,s_3,s_2;l-1)+H(s_3,s_1,s_2;l-1)+H(s_3,s_1+s_2;l-1)\\
&+H(s_1+s_3,s_2;l-1)+H(s_1,s_2+s_3;l-1)+H(s_1+s_2+s_3;l-1)\bigg]\\
\sum_{j=1}^{p-1}\bigg(\sum_{l=j+1}^{p-1}\frac{1}{l^{s_4}}\bigg)\frac{H_j^{(s_1)}H_j^{(s_3)}}{j^{s_2}}=&H(s_1,s_3,s_2,s_4)+H(s_3,s_1,s_2,s_4)+H(s_3,s_1+s_2,s_4)+H(s_1+s_3,s_2,s_4)\\
&+H(s_1,s_2+s_3,s_4)+H(s_1+s_2+s_3,s_4)\\
\sum_{j=1}^{p-1}\big(H(s_4)-H_j^{(s_4)}\big)\frac{H_j^{(s_1)}H_j^{(s_3)}}{j^{s_2}}=&H(s_1,s_3,s_2,s_4)+H(s_3,s_1,s_2,s_4)+H(s_3,s_1+s_2,s_4)+H(s_1+s_3,s_2,s_4)\\
&+H(s_1,s_2+s_3,s_4)+H(s_1+s_2+s_3,s_4).
\end{align*}
Hence, \ref{principallengthfoureq} yields immediately.
\end{proof}

\begin{corollary}For $p\geq 7$ we have
\begin{equation}
\sum_{j=1}^{p-1}\frac{H_j^{3}}{j}\equiv \frac{3}{2}pB_{p-5} {\pmod {p^2}}.
\end{equation}
\end{corollary}
\begin{proof}
From the previous Theorem, we have modulo $p^2$
\begin{align*}
-\sum_{j=1}^{p-1}\frac{H_j^{3}}{j}\equiv&2H(\left\lbrace 1\right\rbrace^4 )+2H(1,2,1)+H(2,1,1)+H(3,1)\\
-\sum_{j=1}^{p-1}\frac{H_j^{3}}{j}\equiv&-\frac{2}{5}pB_{p-5}-\frac{9}{5}pB_{p-5}+\frac{3}{5}pB_{p-5}+\frac{1}{10}pB_{p-5}=-\frac{3}{2}pB_{p-5}.
\end{align*}
Note that another proof can be found in \cite{mistrovic2}.
\end{proof}

\begin{lemma}\cite{tauraso1}
Let $a,b$ be non-negative integers and a prime $p > 2a + 2b + 3$. Then
\begin{equation}
H(\left\lbrace 2\right\rbrace^a,3,\left\lbrace 2\right\rbrace^b )\equiv \frac{(-1)^{a+b}(a-b)}{(a+1)(b+1)}{2a+2b+2\choose 2a+1}B_{p-2a-2b-3} \pmod p,
\end{equation}
And for primes such that $p > 2a + 2b + 1$, we have
\begin{equation}
H(\left\lbrace 2\right\rbrace^a,1,\left\lbrace 2\right\rbrace^b )\equiv \frac{4(-1)^{a+b}(a-b)(1-4^{-a-b})}{(2a+1)(2b+1)}{2a+2b\choose 2a}B_{p-2a-2b-1} \pmod p.
\end{equation}
\end{lemma}

\begin{corollary}For $p\geq 11$ we have
\begin{equation}
-\frac{1}{13}\sum_{j=1}^{p-1}\frac{(H_j^{(2)})^2H_j^{(3)}}{j^{2}}\equiv \frac{3}{83} \sum_{j=1}^{p-1}\frac{(H_j^{(2)})^3}{j^{3}}\equiv B_{p-9} \pmod p,
\end{equation}
\begin{equation}
-\frac{8}{21}\sum_{j=1}^{p-1}\frac{(H_j^{(2)})^2H_j}{j^{2}}\equiv\frac{1}{3}\sum_{j=1}^{p-1}\frac{(H_j^{(2)})^3}{j} \equiv B_{p-7}\pmod p.
\end{equation}
\end{corollary}

\begin{proof}

For the first congruence, we apply the previous Lemma, Theorems \ref{principallengthfour}, \ref{weighttwo} and \ref{weightthree} and obtain
\begin{align*}
\sum_{j=1}^{p-1}\frac{(H_j^{(2)})^2H_j^{(3)}}{j^{2}}&=-2H(\left\lbrace 2\right\rbrace^3,3)-2H(2,4,3)-H(4,2,3)-H(6,3)       \\
&\equiv \big[+12-\frac{40}{3}-\frac{35}{3}-0 \big]B_{p-9}=-13B_{p-9},
\end{align*}
also
\begin{align*}
\sum_{j=1}^{p-1}\frac{(H_j^{(2)})^3}{j^{3}}&=-2H(\left\lbrace 2 \right\rbrace^2,3,2)-2H(2,5,2)-H(5,2,2)-H(7,2)       \\
&\equiv \big[ +\frac{56}{3}-0+9-0\big]B_{p-9}=\frac{83}{3}B_{p-9}.
\end{align*}
The second part of the corollary is proved in a quite similar manner.

\end{proof}

\bibliographystyle{unsrt}

\end{document}